\documentclass[11pt,a4paper]{amsart}
\usepackage{amsmath,amsthm,amssymb,latexsym}
\usepackage{graphicx}

\newcommand{\re}{\text{\rm Re\,}}
\newcommand{\im}{\text{\rm Im\,}}

\newcommand{\bd}{{\mathbb{D}}}

\newcommand{\br}{{\mathbb{R}}}

\newcommand{\bz}{{\mathbb{Z}}}
\newcommand{\bc}{{\mathbb{C}}}

\newcommand{\csp}{{\mathcal{S}_p}}

\renewcommand{\a}{\alpha}
\renewcommand{\b}{\beta}
\renewcommand{\l}{\lambda}

\newcommand{\s}{\sigma}

\renewcommand{\r}{\rho}

\newcommand{\dd}{\Delta}
\renewcommand{\o}{\omega}

\renewcommand{\O}{\Omega}

\newcommand{\g}{\gamma}

\newcommand{\eps}{\varepsilon}

\newcommand{\nt}{\noindent}

\newcommand{\ovl}{\overline}

\newcommand{\lp}{\left(}
\newcommand{\rp}{\right)}

\DeclareMathOperator{\di}{\rm d}

\allowdisplaybreaks
\numberwithin{equation}{section}

\newtheorem{theorem}{Theorem}[section]

\newtheorem{lemma}[theorem]{Lemma}

\theoremstyle{definition}

\newtheorem{remark}[theorem]{Remark}

\begin{document}

\title[Infinite band Schr\"odinger operators]
{On complex perturbations of infinite band Schr\"odinger
operators}
\author[L. Golinskii, S. Kupin]{L. Golinskii and S. Kupin}

\address{Mathematics Division, Institute for Low Temperature Physics and
Engineering, 47 Lenin ave., Kharkov 61103, Ukraine}
\email{golinskii@ilt.kharkov.ua}

\address{IMB, Universit\'e Bordeaux 1, 351 cours de la Lib\'eration, 33405 Talence Cedex, France}
\email{skupin@math.u-bordeaux1.fr}

\date{\today}


\maketitle

\begin{center}
{\it Dedicated to Yu. M. Berezanskii\\
on occasion of  his 90-th birthday}

\bigskip
\end{center}

\begin{abstract}
Let $H_0=-\frac{d^2}{dx^2}+V_0$ be an infinite band Schr\"odinger operator on $L^2(\br)$ with
a real-valued potential $V_0\in L^\infty(\br)$. We study its complex perturbation $H=H_0+V$,
defined in the form sense, and obtain the Lieb--Thirring type inequalities for the rate of
convergence of the discrete spectrum of $H$ to the joint essential spectrum. The assumptions
on $V$ vary depending on the sign of $\re V$.
\end{abstract}

\section*{Introduction}
\label{s0}

Different characteristics of the distribution of the discrete spectrum for non-self-adjoint
perturbations of model differential self-adjoint operators, e.g., a Laplacian on $\br^d$,
a discrete Laplacian on $\bz^d$, etc., were studied in a number of papers
(see Frank--Laptev--Lieb--Seiringer \cite{fla}, Borichev--Golinskii--Kupin \cite{bgk},
Demuth--Hansmann--Katriel \cite{dhk}). This paper focuses on complex perturbations of one dimensional
Schr\"odinger operators with infinite band spectrum and certain behavior of the lengths of its gaps
(the case of finite band Schr\"odinger operators was studied in \cite{gk}).

So, consider a real-valued measurable function $V_0$ on $\br$ and denote by $M_{V_0}$ a maximal
multiplication operator by $V_0$. The standing assumption is that the Schr\"odinger operator
\begin{equation}\label{hill}
H_0=-\dd+M_{V_0}, \qquad \dd:=\frac{d^2}{dx^2}.
\end{equation}
is self-adjoint, $H_0^*=H_0$, and its spectrum $\s(H_0)$ is infinite band, i.e.,
\begin{equation}\label{infband}
\s(H_0)=\s_{ess}(H_0)=I=\bigcup^\infty_{k=1} [a_k, b_k], \qquad a_k\to +\infty.
\end{equation}
We say that the gaps are relatively bounded if
\begin{equation}\label{irelbound}
r=r(I):=\sup_k \frac{r_k}{b_k}<\infty, \qquad r_k:=a_{k+1}-b_k
\end{equation}
is the length of $k$' gap in \eqref{infband}. A typical example here is the Hill operator with
a periodic potential (see \cite[Section XIII.16]{rs4}). It is well known (see \cite{maos})
that $r_k\to 0$ as $k\to\infty$ for potentials $V_0$ from $L^2$ on a period, so \eqref{irelbound}
obviously holds for such potentials.

Furthermore, consider the form sum
\begin{equation}\label{perturbation}
H=H_0+M_V,
\end{equation}
where $V$ is a complex-valued potential. If $M_V$ is a relatively compact perturbation of $H_0$,
that is, $\rm{dom}(M_V)\supset \rm{dom}(H_0)$, and $M_V(H_0-z)^{-1}$ is a compact operator for
$z\in\r(H_0)$, then, by the celebrated theorem of Weyl (see, e.g., \cite[Section IV.5.6]{kato}), $\s_{ess}(H)=\s_{ess}(H_0)$ and
$$ \s(H)=I\,\dot\cup\,\s_d(H) $$
(disjoint union), the discrete spectrum $\s_d(H)$ of $H$, i.e., the set of isolated eigenvalues
of finite algebraic multiplicity, can accumulate only on $I$. The main goal of the paper is to
obtain certain quantitative bounds for the rate of this accumulation.

The assumption on the background $V_0$ looks as follows
\begin{equation}\label{hyp1}
V_0\ge0, \qquad V_0\in L^\infty(\br).
\end{equation}
As for the perturbation $V$ the conditions will vary depending on the sign of $\re V$. Precisely, 
for general $V$'s we assume that
\begin{equation}\label{hyp2}
V\in L^p(\br), \qquad p\ge2,
\end{equation}
and for accretive perturbations with $\re V\ge0$ we put
\begin{equation}\label{hyp3}
V\in L^p(\br), \qquad p>1.
\end{equation}
Under assumptions \eqref{hyp1}--\eqref{hyp3} $H$ is a well-defined, closed and $m$-sectorial operator, 
and there is a number $\o_1\le 0$ such that
\begin{equation}\label{inclus}
\s(H)\subset \ovl{N(H)}\subset \{z: \re z\ge\o_1\},
\end{equation}
where $N(H)=\{(Hf,f): f\in {\rm dom}(H), \ \|f\|_2\le 1\}$ is the numerical range of $H$ (see, e.g.,
\cite[Chapter VI]{kato}). Moreover, $H$ appears to be a relatively compact (even $\csp$) perturbation 
of $H_0$, $\csp$ being the Schatten--von Neumann class of compact operators.

Denote by $\di(z,E)$ the Euclidian distance from a point $z\in\bc$ to a set $E\subset\br$.

\begin{theorem}\label{t1} Let $H_0$ be the Schr\"odinger operator \eqref{hill} with $V_0$ satisfying 
$\eqref{hyp1}$. Assume that $H_0$ is the infinite band operator with the spectrum
$$ \s(H_0)=\s_{ess}(H_0)=\bigcup^{\infty}_{k=1}[a_k,b_k], \ \  0\le a_1<b_1<a_2<b_2<\ldots, 
 \ \ a_n\to+\infty, $$
and the lengths of gaps are relatively bounded $\eqref{irelbound}$. Then for the perturbation $H$ $\eqref{perturbation}$
with $V$ $\eqref{hyp2}$ and for each $\o<\o_1$ $\eqref{inclus}$ the following Lieb--Thirring type 
inequality
\begin{equation}\label{lith1}
\sum_{z\in\s_d(H)} \frac{\di^p(z,I)}{(|z-\o|+|\o|)^{2p}}\le
\frac{C(p,I)\,\|V\|_p^p}{(\o_1-\o)^p|\o|^{p-1/2}}\left(1+\frac{\|V_0\|_\infty}{a_1+|\o|}\right)^p,
\end{equation}
holds, where a positive constant $C(p,I)$ depends on $p$ and $I=\s(H_0)$.
\end{theorem}

\begin{remark}
If we take $\o<\o_1-1$, bound \eqref{lith1} can be simplified. Indeed, now $\o_1-\o>1$,
$$ |z-\o|<|\o|(1+|z|), \qquad 1<a_1+|\o|<|\o|(1+a_1), $$
and so
\begin{equation}\label{lith2}
\sum_{z\in\s_d(H)} \frac{\di^p(z,I)}{(1+|z|)^{2p}}
\le C(p,I)|\o|^{p+1/2}\,(1+\|V_0\|_\infty)^p\,\|V\|_p^p.
\end{equation}
\end{remark}

There is an elementary way to specify $\o$ and eliminate it from the final expression. 
The price we pay is an additional factor in the right hand side.

\begin{theorem}\label{t2}
Under assumptions $\eqref{hyp1}$, $\eqref{hyp2}$
\begin{equation}\label{e1041}
\sum_{z\in\s_d(H)}\frac{\di^p(z, I)}{(1+|z|)^{2p}}\le C(p,I) (1+\|V_0\|_\infty)^p\,(1+\|V\|_p)^{p\frac{2p+1}{2p-1}}\,
\|V\|_p^p,
\end{equation}
where a positive constant $C(p,I)$ depends on $p$ and $I=\s(H_0)$.
\end{theorem}

Denote $\bd_+=\{|z|<1\}$, $\bd_-=\{|z|\ge1\}$.
\begin{theorem}\label{t3}
Let $\re V\ge0$. Under assumptions $\eqref{hyp1}$, $\eqref{hyp3}$ the following Lieb--Thirring 
type inequality holds for each $0<\eps<1$
\begin{equation}\label{lith3}
\sum_{z\in\s_d(H)\cap\bd_+} \frac{\di^p(z, I)}{|z|^{1/2-\eps}}+\sum_{z\in\s_d(H)\cap\bd_-} 
\frac{\di^p(z, I)}{|z|^{1/2+\eps}} \le C(p,I,\eps)\|V\|_p^p.
\end{equation}
\end{theorem}

\begin{remark}
The only reason we restricted ourselves with the case of one dimensional Schr\"odinger operator 
$H_0$ as a background is that the class of multidimensional Schr\"odinger operators with spectra 
\eqref{infband} is not well understood. Our technique works for any dimension $d\ge1$, and the 
corresponding problem will be elaborated on elsewhere.
\end{remark}

\section{Distortion for linear fractional transformations}
\label{s1}

The main analytic tool in the proof of Theorem \ref{t1} is the following distortion lemma for 
linear fractional transformations of the form
\begin{equation}\label{lft}
\l_\o(z):=\frac1{z-\o}\,, \quad \o\in\br.
\end{equation}
The argument here is quite elementary (though, rather lengthy).


\begin{figure}[b]
\includegraphics[width=14cm]{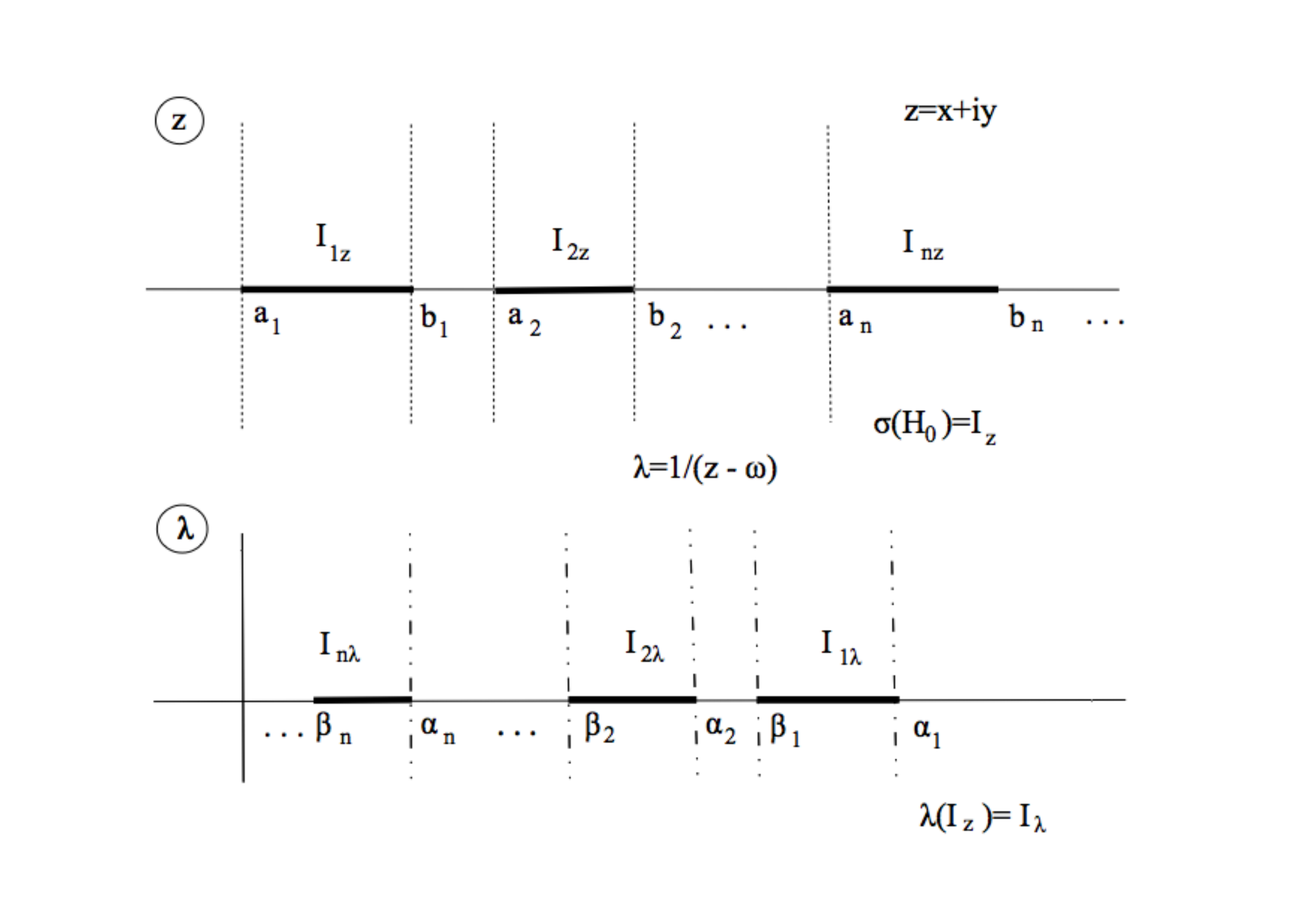}
\caption{Sets $I=\s(H_0)$ and $\l_\o(I)$ with map $\l_\o(z)=\frac 1{z-\o}$.}\label{f1}
\end{figure}

\begin{lemma}\label{l1}
Let
 \begin{equation}\label{infiniteband}
 I=I_z=\bigcup^{\infty}_{k=1}[a_k,b_k], \quad
 0\le a_1<b_1<a_2<b_2<\ldots, \quad a_n\to+\infty, 
 \end{equation}
and let $\l_\o(I)=I_\l$ be its image under the linear fractional transformation $\eqref{lft}$
\begin{equation*} 
 \l_{\o}(I)=I_\l=\bigcup^{\infty}_{k=1}[\b_k(\o), \a_k(\o)], \ \ \b_k(\o)=\frac1{b_k-\o}\,, \ \  \a_k(\o)=\frac1{a_k-\o}\,.
\end{equation*}
Then for $\o<a_1$ the following bounds hold:

for $\re z<a_1$ or $\re z\in I$
\begin{equation}\label{distor1}
\frac{\di(\l_\o(z),\l_\o(I)}{\di(z,I_z)}>\frac1{3|z-\o|(|z-\o|+a_1-\o)}\,;
\end{equation}

for $b_k<\re z<a_{k+1}$, \ $k=1,2,\ldots$
\begin{equation}\label{distor2}
\frac{\di(\l_\o(z),\l_\o(I)}{\di(z,I_z)}
\ge\frac1{2|z-\o|^2}\,\left(1+\frac{a_{k+1}-b_k}{b_k-\o}\right)^{-1}\,.
\end{equation}
Moreover, if $\o\le0$ and the gaps are relatively bounded $\eqref{irelbound}$,
then the unique bound is valid
\begin{equation}\label{distor3}
\frac{\di(\l_\o(z),\l_\o(I)}{\di(z,I_z)}\ge
\frac1{5(1+r(I))}\,\frac1{|z-\o|(|z-\o|+a_1-\o)}, \ \
z\in\bc\backslash I.
\end{equation}
\end{lemma}
\begin{proof}
With no loss of generality we can assume that $a_1>0$.

We begin with the case $\o=0$ and put $\l_0=\l=z^{-1}$. If
$z=x+iy$ and $x=\re z\le0$, then $\re\l=x|z|^{-2}\le0$ and so
\begin{equation}\label{below1}
\frac{\di(\l,I_\l)}{\di(z,I_z)}=\frac{|\l|}{|z-a_1|}=\frac1{|z||z-a_1|}\ge
\frac1{|z|(|z|+a_1)}\,.
\end{equation}
Similarly, if $x\in I_z$, then $x\ge a_1$ and
$$ 0<\re\l=\frac{x}{|z|^2}\le\frac1{x}\le
a_1^{-1}=\a_1, \quad \di(\l,[0,\a_1])=|\im\l|=\frac{|y|}{|z|^2}\,.
$$
Since now $\di(z,I_z)=|y|$, we have
\begin{equation}\label{below2}
\frac{\di(\l,I_\l)}{\di(z,I_z)}\ge
\frac{\di(\l,[0,\a_1])}{\di(z,I_z)}=\frac1{|z|^2}>\frac1{|z|(|z|+a_1)}\,.
\end{equation}

\smallskip

Consider now the case when $x=\re z\notin I_z$. Fix $x$ in $k$'s
gap,
\begin{equation}\label{gapsforz}
b_k<x<a_{k+1}, \qquad k=k(x)=0,1,\ldots \end{equation} (we put
$b_0=0$ and treat $(b_0,a_1)$ as a number zero gap). Then
$$ \di(z,I_z)=\min(|z-b_k|,|z-a_{k+1}|), \quad k=1,2,\ldots, \quad \di(z,I_z)=|z-a_1|, 
\quad k=0. $$
Define two sets of positive numbers
$$ u_j=u_j(x), \qquad v_j=v_j(x), \qquad j=k+1,k+2,\ldots $$ by equalities
\begin{equation*}
\re(\l(x+iu_j))=\frac{x}{x^2+u_j^2}=\a_j, \qquad
\re(\l(x+iv_j))=\frac{x}{x^2+v_j^2}=\b_j,
\end{equation*}
or, equivalently,
$$ u_j(x)=\sqrt{x(a_j-x)}, \qquad v_j(x)=\sqrt{x(b_j-x)}. $$
We also put $v_k=0$, so
$$ 0=v_k<u_{k+1}<v_{k+1}<u_{k+2}<v_{k+2}<\ldots, \quad u_n,\ v_n\to\infty, \ n\to\infty. $$

While the point $z$ traverses the line $x+iy$, $y\in\br$, its
image $\l(z)$ describes a circle with diameter $[0,1/x]$. We
discern the following two cases.

{\bf Case 1}. Assume that $\l$ lies over the ``gaps for $\l$''.
For each $k=0,1,\ldots$ there are two options for $\l$: interior
gaps
\begin{equation}\label{intgap}
\re\l\in(\a_{j+1},\b_j)\Longleftrightarrow v_j<|y|<u_{j+1}, \qquad
j=k+1,k+2,\ldots
\end{equation}
and the rightmost gap
\begin{equation}\label{rightgap}
\re\l\in(\a_{k+1},1/x)\Longleftrightarrow 0<|y|<u_{k+1}.
\end{equation}
For gaps \eqref{intgap} we have
 \begin{equation}\label{intgap1}
\di(\l,I_\l)=\min(|\l-\a_{j+1}|,
|\l-\b_j|)=\frac1{|z|}\min\left(\frac{|z-a_{j+1}|}{a_{j+1}},
\frac{|z-b_j|}{b_j}\right)\,.
\end{equation}
Define an auxiliary function $h$ on the right half-line
$$
h(t)=h(t,z):=\frac{|z-t|}{t}=\sqrt{\Bigl(\frac{x}{t}-1\Bigr)^2+y^2},
\quad t>0. $$
 Clearly, $h$ is monotone increasing on $(x,+\infty)$ and
 decreasing on $(0,x)$ with the minimum $h(x)=|y|$. Hence
 \eqref{intgap1} and \eqref{gapsforz} give
\begin{equation*}
\begin{split}
\di(\l,I_\l) &=\frac{\min(h(a_{j+1},z), h(b_j,z))}{|z|}\ge
 \frac{h(b_{j},z)}{|z|}\ge \frac{h(b_{k+1},z)}{|z|} \\ &\ge\frac{h(a_{k+1},z)}{|z|}=
 \frac{|z-a_{k+1}|}{a_{k+1}|z|}\,.
 \end{split}
 \end{equation*}
 Since by \eqref{gapsforz} $\di(z,I_z)\le |z-a_{k+1}|$, we see that
 \begin{equation}\label{bound1}
 \frac{\di(\l,I_\l)}{\di(z,I_z)}\ge \frac1{a_{k+1}|z|}\,.
 \end{equation}

For gaps \eqref{rightgap} let first $k\ge1$. Then as above in
\eqref{intgap1}
 \begin{equation*}
\di(\l,I_\l)=\frac1{|z|}\min\left(\frac{|z-a_{k+1}|}{a_{k+1}},
\frac{|z-b_k|}{b_k}\right)\,,
\end{equation*}
but it is not clear now which term prevails. If
$|z-a_{k+1}|\le|z-b_k|$ then $\di(z,I_z)=|z-a_{k+1}|$ and
$$
\frac{\di(\l,I_\l)}{\di(z,I_z)}=\frac1{|z|}\min\left(\frac1{a_{k+1}},
\frac{|z-b_k|}{b_k|z-a_{k+1}|}\right)=\frac1{a_{k+1}|z|}\,.
$$
Otherwise $|z-a_{k+1}|>|z-b_k|$ implies
$$
\frac{\di(\l,I_\l)}{\di(z,I_z)}=\frac1{|z|}\min\left(\frac1{b_{k}},
\frac{|z-a_{k+1}|}{a_{k+1}|z-b_{k}|}\right)\ge\frac1{a_{k+1}|z|}\,.
$$

Next, for $k=0$ one has $0<x<a_1$, and in case \eqref{rightgap}
$$ \di(\l,I_\l)=|\l-\a_1|=\frac{|z-a_1|}{a_1|z|}\,, \quad
\di(z,I_z)=|z-a_1|, $$ and so
\begin{equation}\label{kzero}
\frac{\di(\l,I_\l)}{\di(z,I_z)}=\frac1{a_1|z|}\,.
\end{equation}
 Finally, in the case of ``gaps for $\l$'' we come to the bound
\begin{equation}\label{gapsforl}
\frac{\di(\l,I_\l)}{\di(z,I_z)}\ge \frac1{a_{k+1}|z|}\,, \qquad
k=0,1,\ldots.
\end{equation}

A modified version of \eqref{gapsforl} will be convenient in the
sequel. For $k\ge1$ in view of $|z|\ge x>b_k$ we have
$$
\frac1{a_{k+1}|z|}\ge\frac{b_k}{a_{k+1}|z|^2} $$ and so for
$k=1,2,\ldots$
\begin{equation}\label{gapsforl1}
\frac{\di(\l,I_\l)}{\di(z,I_z)}\ge \frac1{|z|^2}\,
\left(1+\frac{a_{k+1}-b_k}{b_k}\right)^{-1}=\frac1{|z|^2}\left(1+\frac{r_k}{b_k}\right)^{-1}\,,
\end{equation}
$r_k=a_{k+1}-b_k$ is the length of $k$'s gap. Similarly, for $k=0$
one has from \eqref{kzero}
\begin{equation}\label{gapsforl2}
\frac{\di(\l,I_\l)}{\di(z,I_z)}\ge \frac1{|z|(|z|+a_1)}\,.
\end{equation}

{\bf Case 2}. Assume that $\l$ lies over the ``bands for $\l$''
\begin{equation}\label{bands1}
\re\l\in[\b_j,\a_{j}]\Longleftrightarrow u_j\le |y|\le v_{j},
\qquad j=k+1,k+2,\ldots.
\end{equation}
Now
$$ \di(\l,I_\l)=|\im\l|=\frac{|y|}{|z|^2}\,, $$
\begin{equation*}
\begin{split}
\di(z,I_z) &\le |z-a_{k+1}|\le
|y|+a_{k+1}-x=|y|+\frac{u_{k+1}^2}{x} \le
|y|\left(1+\frac{u_{k+1}}{x}\right)
\\ &=|y|\left(1+\sqrt{\frac{a_{k+1}-x}{x}}\right),
\end{split}
\end{equation*}
so that
\begin{equation}\label{bands2}
\frac{\di(\l,I_\l)}{\di(z,I_z)}\ge \frac1{|z|^2}\,
\left(1+\sqrt{\frac{a_{k+1}}{x}-1}\right)^{-1}\,.
\end{equation}

For $k\ge1$ (interior gap for $z$) inequality \eqref{bands2} can
be simplified in view of $x>b_k$
\begin{equation}\label{bands3}
\frac{\di(\l,I_\l)}{\di(z,I_z)}\ge \frac1{|z|^2}\,
\left(1+\sqrt{\frac{r_k}{b_k}}\right)^{-1}\,.
\end{equation}
Let now $k=0$, i.e., $0<x=\re z<a_1$. In our case
$\di(z,I_z)=|z-a_1|$ and
$$ |y|\ge u_1=\sqrt{x(a_1-x)}. $$
If $|y|\ge 2x$ then $|y|\ge \frac23|z|$ and so
\begin{equation}\label{bands4}
\frac{\di(\l,I_\l)}{\di(z,I_z)}=\frac{|y|}{|z|^2|z-a_1|}\ge
\frac23 \frac1{|z|(|z|+a_1)}\,.
\end{equation}
Otherwise, $|y|<2x$ implies
$$ 2\sqrt{x}>\sqrt{a_1-x}, \qquad x>\frac{a_1}5. $$
It follows now from \eqref{bands2} with $k=0$ that
\begin{equation}\label{bands5}
\frac{\di(\l,I_\l)}{\di(z,I_z)}\ge \frac1{3|z|^2}>\frac13
\frac1{|z|(|z|+a_1)}\,.
\end{equation}

We can summarize the results obtained above in the following two
bounds from below. A combination of \eqref{below1},
\eqref{below2}, \eqref{gapsforl2} and \eqref{bands5} gives
\begin{equation}\label{final1}
\frac{\di(\l,I_\l)}{\di(z,I_z)}>\frac1{3|z|(|z|+a_1)}\,, \qquad
\re z<a_1 \ \ {\rm or} \ \re z\in I_z.
\end{equation}
A combination of \eqref{gapsforl1} and \eqref{bands3} provides
\begin{equation}\label{final2}
\begin{split}
\frac{\di(\l,I_\l)}{\di(z,I_z)} &\ge\frac1{\g_k|z|^2}\,, \qquad
\g_k =\max\left\{1+\frac{r_k}{b_k},\
1+\sqrt{\frac{r_k}{b_k}}\right\}\,, \\ b_k&<\re z<a_{k+1}, \quad
k=1,2,\ldots.
\end{split}
\end{equation}
Since $\g_k<2(1+r_k/b_k)$, the latter can be written as
\begin{equation}\label{final3}
\frac{\di(\l,I_\l)}{\di(z,I_z)}
\ge\frac1{2|z|^2}\,\left(1+\frac{r_k}{b_k}\right)^{-1}\,, \quad
b_k<\re z<a_{k+1}, \quad k=1,2,\ldots.
\end{equation}

To work out the general case $\o\not=0$ and prove \eqref{distor1}
and \eqref{distor2}, it remains only to shift the variable and
apply the results just obtained to the shifted sequence of bands
$$ I_z(\o)=\bigcup_{k\ge1}[a_k-\o,b_k-\o]. $$
The final statement follows from a simple observation
$$ \frac{r_k}{b_k-\o}\le\frac{r_k}{b_k}\le r. $$
The proof is complete.
\end{proof}

\section{Lieb--Thirring type inequalities}
\label{s2}

The key ingredient of the proof of our main statements is the following result of Hansmann 
\cite[Theorem 1]{han1}. Let $A_0=A_0^*$ be a bounded self-adjoint operator on the Hilbert space, 
and let $A$ be a bounded operator with $A-A_0\in\csp$, $p>1$. Then
\begin{equation}\label{hans}
\sum_{\l\in\s_d(A)} \di^p(\l, \s(A_0))\le K_p\|A-A_0\|_{\csp}^p,
\end{equation}
$K_p$ is an explicit (in a sense) constant, which depends only on $p$.
We set
$$
A_0(\o)=R(\o,H_0)=(H_0-\o)^{-1}, \qquad A(\o)=R(\o,H)=(H-\o)^{-1},
$$
$\o$ is defined above, and $\o\in\rho(H_0)\cap\rho(H)$ in view of \eqref{infband} and \eqref{inclus}.

Let $\l=\l_\o(z)=(z-\o)^{-1}$. The Spectral Mapping Theorem implies that
$$
\l\in\s_d(A(\o)) \quad \lp \l\in\s(A_0(\o)) \rp\quad
\Longleftrightarrow  \quad z\in\s_d(H) \quad \lp z\in\s(H_0) \rp.
$$

\medskip
\nt {\it Proof of Theorem \ref{t1}.}
The second resolvent identity reads
\begin{equation*}
R(z,H)-R(z,H_0)=-R(z,H)M_V R(z, H_0), \qquad z\in\rho(H)\cap\rho(H_0).
\end{equation*}
We wish to show that this difference belongs to $\csp$ and to obtain the bound for its $\csp$-norm.

First, we have
\begin{equation}\label{e3}
\begin{split}
W &=W(z):=M_V R(z, H_0)=M_V(-\dd-z)^{-1}(-\dd-z) (H_0-z)^{-1}\\  
&= M_V(-\dd-z)^{-1}\, (1-M_{V_0} (H_0-z)^{-1}),
\end{split}
\end{equation}
and so
$$ \|W(z)\|_\csp\le \|M_V\,R(z,-\dd)\|_\csp \|I-M_{V_0}\,R(z,H_0)\|, 
\quad z\in\rho(H)\cap\rho(H_0). $$
It is clear that
$$ \|I-M_{V_0}\,R(z,H_0)\|\le 1+\frac{\|V_0\|_\infty}{\di(z,I)}=
1+\frac{\|V_0\|_\infty}{|a_1-z|}, \quad \re z<0. $$

Next, write
$$ M_V(-\dd-z)^{-1}=V(x)g_z(-i\nabla), \qquad g_z(x)=(x^2-z)^{-1}, \quad x\in\br. $$
By \cite[Theorem 4.1]{si1}
$$ \|M_V(-\dd-z)^{-1}\|_\csp \le (2\pi)^{-1/p}\|V\|_p\, \|g_z\|_p, \qquad p\ge2. $$
Since $ 2|t-z|^2\ge (t+|z|)^2$ for $t\ge0$ and $\re z<0$, we have
$$\|g_z\|_p\le\sqrt{2}\|g_{-|z|}\|_p $$
and so
\begin{equation}\label{e105}
\begin{split}
&{}\|M_V(-\dd-z)^{-1}\|_\csp \le \frac{C_1}{|z|^{1-1/2p}}\,\|V\|_p, \\
C_1&=C_1(p)=\sqrt{2}\left\{\frac1{2\pi}\int_{-\infty}^\infty \frac{dx}{(x^2+1)^p}\right\}^{1/p}.
\end{split}
\end{equation}
Thus,
\begin{equation}\label{boundw}
\|W(z)\|_\csp\le \frac{C_1(p)\|V\|_p}{|z|^{1-1/2p}}\,\left(1+\frac{\|V_0\|_\infty}{|a_1-z|}\right), 
\quad \re z<0.
\end{equation}

We put $z=\o<\o_1$. Relation \eqref{inclus} implies in view of \cite[Theorem~V.3.2]{kato}
\begin{equation}\label{resnum}
\|R(\o, H)\|\le \frac 1{\di(\o, \ovl{N(H)})}\le \frac1{\o_1-\o}\,,
\end{equation}
and the combination of \eqref{boundw} and \eqref{resnum} leads to the following bound for each 
$\o<\o_1$
\begin{equation}\label{upbound2}
\begin{split}
\|R(\o, H)-R(\o, H_0)\|_\csp^p &\le \|R(\o,H)\|^p\,\|W(\o)\|_\csp^p \\
&\le \frac{C_2(p)\,\|V\|_p^p}{(\o_1-\o)^p|\o|^{p-1/2}}
\left(1+\frac{\|V_0\|_\infty}{a_1+|\o|}\right)^p.
\end{split}
\end{equation}

\smallskip

We go back to \eqref{hans} with
$$ A_0=A_0(\o)=R(\o,H_0), \qquad A=A(\o)=R(\o,H),
$$
so by the Spectral Mapping Theorem, in the notation of Lemma \ref{l1} we have
\begin{equation}\label{e102}
\sum_{\l\in\s_d(A(\o))} \di^p(\l, \s(A_0(\o))=\sum_{z\in\s_d(H)} \di^p(\l_\o(z), \l_\o(I))
\le K_p\,\|R(\o, H)- R(\o, H_0)\|^p_\csp.
\end{equation}
We apply Lemma \ref{l1} in the form \eqref{distor3} to obtain
\begin{equation*}
\sum_{z\in\s_d(H)} \frac{\di^p(z,I)}{|z-\o|^p (|z-\o|+a_{1}+|\o|)^p}\le
\frac{C_3(p,I)\,\|V\|_p^p}{(\o_1-\o)^p|\o|^{p-1/2}}\left(1+\frac{\|V_0\|_\infty}{a_1+|\o|}\right)^p,
\end{equation*}
and \eqref{lith1} follows. The proof is complete. \hfill $\Box$

\medskip

The proof of Theorem \ref{t1} shows that bound \eqref{lith1} essentially depends on the 
parameter $\o$. Roughly speaking, it comes from a bound from below of $\inf \re \s(H)$, and so it 
seems to be rather important to estimate this quantity in terms of $V_0$ and $V$ only.

\medskip
\nt {\it Proof of Theorem \ref{t2}.} \
Put
$$ \O=\{\re z<0\}\bigcap \left\{|a_1-z|>(1+\|V_0\|_{\infty})\right\}\bigcap \left\{|z|^{1-1/2p}> 4C_1(p)(1+\|V\|_p)\right\}, $$
$C_1$ is defined in \eqref{e105}. We wish to show that $\O\subset \rho(H_0)\cap\rho(H)$. Indeed, 
by \eqref{boundw},
\begin{equation}\label{winf}
\|W(z)\|_\infty\le\|W(z)\|_{\csp}\le \frac{\|V\|_p}{2(1+\|V\|_p)}<\frac12, \quad z\in\O,
\end{equation}
so $I+W(z)$ is invertible and $\|(I+W(z))^{-1}\|<2$. An application of the identity 
$H-z=(1+W(z))(H_0-z)$ completes the proof of our claim.

Next, write the difference of the resolvents in another way
$$ R(z,H)-R(z,H_0)=-R(z,H_0)(1+W(z))^{-1}\,W(z) $$
to obtain for $z\in\O$
\begin{equation*}
\begin{split}
\|R(z,H)-R(z,H_0)\|_{\csp} &\le \|R(z,H_0)\|\,\|(1+W(z))^{-1}\|\,\|W(z)\|_{\csp} \\
&\le \frac{\|V\|_p}{|a_1-z|(1+\|V\|_p)}\le \frac{\|V\|_p}{(1+\|V_0\|_\infty)(1+\|V\|_p)}\,.
\end{split}
\end{equation*}

It is clear by the definition of $\O$ that if $t\in\O$, $t<0$, then $\{\re z\le t\}\subset\O$.
Take $z=\o'<0$ so that
\begin{equation}\label{omega11}
\frac{|\o'|}2=\frac{a_1}2+1+\|V_0\|_{\infty}+\lp 4C_1(1+\|V\|_p)\rp^{1/(1-1/2p)}.
\end{equation}
It is easy to check that $\{z:\re z<\frac{\o'}2\}\subset\O$, so, in particular, $\o'\in\O$ and hence
$$ \|R(\o',H)-R(\o',H_0)\|_{\csp}\le \frac{\|V\|_p}{(1+\|V_0\|_\infty)(1+\|V\|_p)}\,. $$
Once again, \eqref{hans} says
\begin{equation}\label{hansres2}
\sum_{\l\in\s_d(A(\o'))} \di^p(\l,\s(A_0(\o')))\le K_p 
\lp\frac{\|V\|_p}{(1+\|V_0\|_\infty)(1+\|V\|_p)}\rp^p
\end{equation}
for $p>1$ and, using Lemma \ref{l1}, we come to
\begin{equation*}\label{e1111}
\sum_{z\in\s_d(H)} \frac{\di^p(z,I)}{|z-\o'|^p (|z-\o'|+a_{1}+|\o'|)^p}\le C_4(p,I) \lp\frac{\|V\|_p}{(1+\|V_0\|_\infty)(1+\|V\|_p)}\rp^p\,.
\end{equation*}
By the choice of $\o'$ \eqref{omega11},  we have $\re z\ge\o'/2$ for $z\in\s_d(H)$, and so
$$ |z-\o'|\ge\frac{|\o'|}2>\frac{|a_{1}|+|\o'|}4\,, $$
and
\begin{equation*}
\begin{split}
|z-\o'|+a_{1}+|\o'| &< 5|z-\o'|, \\
|z-\o'|(|z-\o'|+a_{1}+|\o'|) &< 5|z-\o'|^2.
\end{split}
\end{equation*}
Next, $|z-\o'|\le (1+|z|)(1+|\o'|)\le 2|\o'|(1+|z|)$ and hence
\begin{equation*}
\begin{split}
\sum_{z\in\s_d(H)} \frac{\di^p(z,I)}{(1+|z|)^{2p}} &\le C_5(p,I)|\o'|^{2p}\lp\frac{\|V\|_p}{(1+\|V_0\|_\infty)(1+\|V\|_p)}\rp^p \\
&\le C_6(p,I)(1+\|V_0\|_{\infty})^{p}(1+\|V\|_p)^{p\frac{2p+1}{2p-1}}\,\|V\|_p^p.
\end{split}
\end{equation*}
The proof is complete. \hfill $\Box$

\medskip\nt {\it Proof of Theorem \ref{t3}.} \
For accretive perturbations we have \newline $\s(H)\subset \{z: \re z\ge0\}$, so one 
can take $\o_1=0$.

The lower bound for the difference of resolvents is the same as above in Theorem \ref{t1}. 
It is a consequence of the result of Hansmann and Lemma \ref{l1}
\begin{equation*}
\|R(\o, H)- R(\o, H_0)\|^p_\csp \ge \sum_{z\in\s_d(H)} \frac{\di^p(z,I)}{(|z-\o|-\o)^{2p}}, \quad p>1.
\end{equation*}
As for the upper bound, we follow the line of reasoning from \cite[Proof of Theorem 3.2]{han2011}, 
where such bound was proved in the case $V_0=0$. As a matter of fact, the argument goes through under 
assumption \eqref{hyp1} as well. At any rate, we have
$$ \|R(-a, H)- R(-a, H_0)\|^p_\csp \le \frac{C_5(p)}{a^{2p-1/2}}\,\|V\|_p^p, \quad a:=-\o>0. $$
Since $\sqrt{2}|z+a|\ge |z|+a$ for $\re z>0$ and $a>0$, we come to the following inequality
\begin{equation}\label{lith4}
\sum_{z\in\s_d(H)} \frac{\di^p(z,I)}{(|z|+a)^{2p}}\le \frac{C_6(p)}{a^{2p-1/2}}\,\|V\|_p^p, 
\quad a>0.
\end{equation}
As in \cite[Proof of Theorem 3.3]{han2011}, we multiply \eqref{lith4} through by $(1+a)^{2\eps}$, 
$\eps>0$, to obtain 
$$ \sum_{z\in\s_d(H)} \frac{\di^p(z,I)\,a^{2p-3/2+\eps}}{(|z|+a)^{2p}(1+a)^{2\eps}}\le
\frac{C_6(p)\|V\|_p^p}{a^{1-\eps}(1+a)^{2\eps}}, \quad a>0, $$
and then integrate the latter inequality with respect to $a\in(0,\infty)$. The proof is complete. 
\hfill $\Box$

\medskip\nt
{\it Acknowledgments.} The authors are grateful to S. Denisov, F. Gesztesy and M. Hansmann for 
valuable discussions on the subject of the article. The paper was prepared during the visit of 
the first author, funded by the IdEx programm of University of Bordeaux,  to the Institute of 
Mathematics of Bordeaux (IMB UMR5251). He  would like to thank this institution for the hospitality.

\end{document}